\documentclass[11pt]{amsart}
\usepackage{amsmath}
\usepackage{amsfonts}
\usepackage{amssymb}

\newcommand\ka{\kappa}
\renewcommand\H{{\mathcal{H}}}
\newcommand\R{{\mathbf{R}}}

\newcommand\E{{\mathcal{E}}}

\theoremstyle{plain}
  \newtheorem{theorem}[subsection]{Theorem}
  
  \newtheorem{proposition}[subsection]{Proposition}
  \newtheorem{lemma}[subsection]{Lemma}

\theoremstyle{remark}
  \newtheorem{remark}[subsection]{Remark}

\theoremstyle{definition}

\begin{document}
\title[]{A remark on the two dimensional water wave problem with surface tension}

\author{Shuanglin Shao}
\address{Department of Mathematics, University of Kansas, Lawrence, KS 66045}
\email{slshao@ku.edu}

\author{Hsi-Wei Shih}
\address{Department of Mathematics, National Cheng Kung University, No. 1 Dasyue Rd., Tainan City 70101, Taiwan}
\email{shihhw@math.ncku.edu.tw}

\vspace{-1in}

\begin{abstract}
We consider the motion of a two-dimensional interface between air (above) and an irrotational, incompressible, inviscid, infinitely deep water (below), with surface tension present. We propose a new way to reduce the original problem into an equivalent quasilinear system which are related to the interface's tangent angle and a quantity related to the difference of tangential velocities of the interface in the Lagrangian and the arc-length coordinates. The new way is relatively simple because it involves only taking differentiation and the real and the imaginary parts. Then if assuming that waves are periodic, we establish a priori energy inequality.
\end{abstract}
\keywords{water waves; surface tension; the a priori energy inequality.}

\maketitle

\section{Introduction}\label{sec:intro}
The two dimensional water wave problem with surface tension concerns the motion of the interface which separates an inviscid, incompressible and irrotational fluid, under the influence of gravity, from the air region of zero density. It is assumed the water region is below the air. Assume that the density of the water is $1$, the gravitational field is $-\textbf{k}$, where $\textbf{k}$ is the unit vector pointing in the upward vertical direction, and at time $t\ge 0$, the free interface is $\Sigma(t)$, and the fluid region occupies $\Omega(t)$. When the surface tension is present, the motion of the fluid is described by
\begin{equation}\label{eq:Euler-coordinate}
\begin{cases}
& v_t+v\cdot \nabla v=-\textbf{k}-\nabla P, \text{ on }\Omega(t), t\ge 0,\\
& \operatorname{div} v=0, \,\operatorname{curl} v=0,  \text{ on }\Omega(t), t\ge 0,\\
& (1,v) \text{ is tangent to the free surface } (t,\Sigma(t)).
\end{cases}
\end{equation}
Here $v$ is the fluid velocity.  $P$ is the fluid pressure, for which we have  $P=\sigma H$ on the interface, where  $H$ denotes the curvature of the interface $\Sigma$
and $\sigma$ the surface tension, which we will set to $\sigma = 1$. We assume that the fluid is of infinite depth. We also assume that the free interface $\Sigma(t)$ is described by
$z = z(\alpha, t)$, where $\alpha\in \R$ is the Lagrangian coordinate, i.e. $z_{t}(\alpha,t)=v(z(\alpha,t),t)$ is the fluid velocity on the interface, $z_{tt}(\alpha, t) =(v_t+v\cdot\nabla v) (z(\alpha,t))$ is the acceleration.

We will regard $\mathbb{R}^2$ as a complex plane and use the same notion for a complex number $z=x+iy$ and a point $z=(x,y)$. Hence  $z(\alpha,t)=(x(\alpha,t),y(\alpha,t)),\, \alpha\in \R,\,$ is a parametrization of the free interface $\Sigma(t)$ (a curve) at time $t$ in the Lagrangian coordinate $\alpha$. The Hilbert transform on $\Sigma(t)$ associated with the Lagrangian parametrization is defined by
\begin{equation}\label{eq:cauchy-transform}
\mathfrak{H}f(\alpha,t)=\frac 1{\pi i}\operatorname{p.v.}\int \frac {f(\beta,t)z_\beta(\beta,t)}{z(\alpha,t)-z(\beta,t)}d\beta.
\end{equation}

Then the system \eqref{eq:Euler-coordinate} is equivalent to the following system on the interface $\Sigma(t)$ \cite{Sijue:1997:2d-local, Sijue:1999:3d-local, Sijue:2009:2d-almost-global}:
\begin{equation}\label{eq:complex-coordinate}
\begin{cases}
& z_{tt}+i=-\nabla P,\\
&\bar{z}_t=\mathfrak{H}\bar{z}_t.
\end{cases}
\end{equation}
We will assume that the non self-intersecting condition on the interface,
\begin{equation}\label{eq-non-intersect-time-0}
\left|\frac{z(\alpha,0)-z(\beta,0)}{\alpha-\beta}\right|>c>0.
\end{equation}
We will assume that the Taylor sign condition  \cite{Taylor:1950:taylor-sign}
\begin{equation}\label{eq-taylor-sign}
a:=-\frac {\partial P}{\partial n} >0,
\end{equation}
where the $\frac {\partial }{\partial n}$ denotes the normal derivative pointing to the air region along the interface. When the surface tension is neglected and the water is of infinite depth, Wu \cite{Sijue:1997:2d-local} proved that the Taylor sign condition \eqref{eq-taylor-sign} holds as long as the surface of the water wave does not intersect (see \eqref{eq-non-intersect-time-0}), see also \cite{Sijue:1999:3d-local} for another intuitive proof by using the maximum principle. The validity of this condition is then used in a series of papers by Wu \cite{Sijue:1997:2d-local, Sijue:1999:3d-local, Sijue:2009:2d-almost-global, Sijue:2009:3d--global}.

In \cite{Ambrose-Masmoudi:2005:zero-surface-tension-2d-water-wave, Ambrose-Masmoudi:2007:LWP-3d-vortex-sheet-with-surface-tension}, Ambrose and Masmoudi studied the two and three dimensional periodic water wave problem with surface tension; they proved that solutions for water wave with surface tension not only locally exist by using energy method, but also converges to those without surface tension when the surface tension coefficient tends to zero. The method in \cite{Ambrose-Masmoudi:2005:zero-surface-tension-2d-water-wave} was inspired by the numerical work \cite{Baker-Meiron-Orszag:1982:generalized-vortex-sheet, HLS:1994:removing-stiffness-interfacial-flows-with-surface-tension, HLS:1997:vortex-sheet-with-surface-tension}, where the authors efficiently compute the vortex sheets in the presence of surface tension. Rather than using Euclidean coordinates, in these works the authors used the curve's tangent angle and (normalized) arc length to parameterize the curve. Ambrose \cite{Ambrose:2003:LWP-vortex-sheet-with-surface-tension} used the same formulation to establish the local wellposedness for the vortex problem with the surface tension. The work by Ambrose and Masmoudi \cite{Ambrose-Masmoudi:2005:zero-surface-tension-2d-water-wave} followed along this line, where the authors treated the water wave problem as one on vorticity on the interface; the Biot-Savart law and the Birkhoff-Rott integrals guarantee that the velocity of the fluids inside the water region and on the free interface can be recovered from vorticity.

Inspired by the reformulations of the water wave problem in the work \cite{Ambrose-Masmoudi:2005:zero-surface-tension-2d-water-wave, Sijue:2009:2d-almost-global}, in this paper we propose a new way to reduce the system \eqref{eq:complex-coordinate} into an equivalent quasilinear system, which seems to be of independent interests.

\begin{proposition}\label{thm-coordinate-changes}
Let $z(\alpha,t)$ be a parametrization of the interface $\Sigma(t)$ in Lagrangian variable $\alpha$ at time $t$. Let $\ka$ be the change of coordinates from Lagrangian to arc-length $s$, and set $u:=\delta_s$ with $\delta:=\ka_t\circ \ka^{-1}$. Then the system \eqref{eq:complex-coordinate} is reduced to the following system in the Lagrangian coordinate,
\begin{equation}\label{eq-coordinate-invariant-form}
\begin{cases}
& \theta_t =i\mathfrak{H}(u\circ \ka) +\phi,\\
& (u\circ \ka)_t=\bigl(\frac {\partial_\alpha}{|z_\alpha|}\bigr)^3\theta -a\frac {\partial_\alpha}{|z_\alpha|}\theta +\psi
\end{cases}
\end{equation} where $\phi,\,\psi$ are error terms.
\end{proposition}

\begin{remark}\label{re-1}
The system \eqref{eq-coordinate-invariant-form} is stated in the Lagrangian coordinate but it is coordinate-invariant because of the coordinate-invariant derivative $\partial_\alpha/|z_\alpha|$. Note that under the change of coordinates from the Lagrangian coordinate to the arc-length coordinate, $$\frac{\partial_\alpha}{|z_\alpha|}\mapsto \partial_s, \partial_t\mapsto \partial_t+\delta \partial_s,$$
we see that the system \eqref{eq-coordinate-invariant-form} recovers that in \cite[Eq. (2.38)]{Ambrose-Masmoudi:2005:zero-surface-tension-2d-water-wave} in the arc-length coordinate:
$$
\begin{cases}
\theta_t &=i\mathcal{H}(u) -\delta u+\bar{\phi},\\
u_t&=\partial_s^3 \theta-a \partial_s \theta+\delta \partial_s u+\bar{\psi}
\end{cases}
$$ for some error terms $\bar{\phi}$ and $\bar{\psi}$, where $\mathcal{H}$ is the standard Hilbert transform. Ambrose and Masmoudi \cite{Ambrose-Masmoudi:2005:zero-surface-tension-2d-water-wave} deduced this system by regarding
the water wave problem as a special vortex sheet problem with zero upper density \cite{Baker-Meiron-Orszag:1982:generalized-vortex-sheet, HLS:1994:removing-stiffness-interfacial-flows-with-surface-tension, HLS:1997:vortex-sheet-with-surface-tension}. Their derivation heavily relied on the machinery of the Birkhoff-Rott integrals and the Biot-Savart law. Our way is relatively simple as it only involves taking differentiation and real and imaginary parts. For details, see Section \ref{sec:new-reduction}.
\end{remark}

\begin{remark}\label{re-2}
Roughly speaking, the system \eqref{eq-coordinate-invariant-form} is on $\ln z_\alpha$ \footnote{This was communicated to us by Sijue Wu.}. This can be seen as follows: $$\ln z_\alpha =\ln |z_\alpha|+i\theta =\ln \ka_\alpha +i\theta.$$ Then after taking time derivative,
$$ \partial_t \ln z_\alpha =\frac {z_{\alpha t}}{z_\alpha} =\frac {\ka_{\alpha t}}{\ka_\alpha}+i\theta_t =(\ka_t\circ \ka^{-1})_s \circ \ka+i\theta_t. $$
Recall that $\delta=\ka_t\circ \ka^{-1}$ and $u=\delta_s$, then
\begin{equation}\label{eq-a46}
\partial_t \ln z_\alpha =u \circ \ka+i\theta_t.
\end{equation}
The real part and the imaginary parts are $u \circ \ka$ and $\theta_t$, which are the quantities studied in \eqref{eq-coordinate-invariant-form}.
\end{remark}

\begin{remark}
Here is a heuristic explanation of the structure of the first equation in \eqref{eq-coordinate-invariant-form}. If the left hand of \eqref{eq-a46} were holomorphic in $\Omega(t)$, then as the real and the imaginary parts of its boundary value on $\Sigma(t)$,  $u\circ \ka$ and $\theta_t$ would be connected by the standard Hilbert transform \cite[Charpter 5, Lemma 2.6]{Stein-Weiss:1971:fourier-analysis}. In the first equation in \eqref{eq-coordinate-invariant-form}, they are connected via $\mathfrak{H}$ up to some error terms.
\end{remark}

Next we restrict our attention to periodic waves: if the position of the surface is given by $(x(\alpha,t), y(\alpha,t))$, then in the complex plane, if $z(\alpha,t)=x(\alpha,t)+iy(\alpha,t)$, the periodic functions we consider satisfy $z(\alpha+2\pi)=z(\alpha)+2\pi$. Our second result is for periodic waves, which is to establish an \emph{a priori} energy inequality for \eqref{eq-coordinate-invariant-form} in the arc-length coordinate. This result was essentially in \cite{Ambrose-Masmoudi:2005:zero-surface-tension-2d-water-wave, Ambrose-Masmoudi:2007:LWP-3d-vortex-sheet-with-surface-tension}. 

\begin{theorem}\label{thm-energy-inequality}
Let $\mathcal{E}(t)$ be some energy for the system in \eqref{eq-coordinate-invariant-form} in the arc-length coordinate. Then $\frac {d \E(t)}{dt} \le C(\E(t))$, where $C(\E(t))$ denotes a polynomial of $\E(t)$.
\end{theorem}

\begin{remark}\label{re-9}
As in \cite{Ambrose-Masmoudi:2005:zero-surface-tension-2d-water-wave}, in order to establish this \emph{a priori} energy inequality, we need certain regularity assumption on the initial data, which is by no means optimal.
\end{remark}

The local wellposedness was proved in \cite{Ambrose:2003:LWP-vortex-sheet-with-surface-tension,Ambrose-Masmoudi:2005:zero-surface-tension-2d-water-wave}. Recently, the global Cauchy problem for two or three dimensional water wave with or without surface tension has been received a lot of attention. We refer interested readers to \cite{Deng-Ionescu-Pausader-Pusateri:2016:gravity-capillary-water-waves-3d,Ifrim-Tataru:2014:water-waves-2d-holomorphic-coordinates,Ionescu-Pusateri:2014:gravity-water-waves-2d} and the references therein.

This paper is organized as follows: In section \ref{sec:notions-lemmas}, we introduce some notation and cite several preliminary lemmas which in the current form are taken mostly from \cite{Sijue:1999:3d-local} and \cite{Sijue:2009:2d-almost-global}. In Section \ref{sec:new-reduction}, we derive \eqref{eq-coordinate-invariant-form} on the curve's tangent angle and $u$ from the system \eqref{eq:complex-coordinate}. In Section \ref{sec:taylor-sign}, we derive an explicit formula for $\partial_\alpha a/|z_\alpha|$ and compute $a-1$ and $a_t$. In Section \ref{sec:pass2arc-length+energy}, we establish \emph{a priori} energy inequality for \eqref{eq-coordinate-invariant-form} if assuming the waves are periodic. Then in Section \ref{sec:serveal-identities}, we discuss the several useful identities which build up the connection between the fluid velocity and the vortex density.

\textbf{Acknowledgments.} The authors would like to thank Sijue Wu for posing this question, and Markus Keel for many helpful discussions during the early preparation of this paper. 

\section{notation and preliminary lemmas}\label{sec:notions-lemmas}
We fix the notation that $\alpha $ denotes the Lagrangian coordinate and $s$ is the arc-length coordinate. The functions $z(\alpha,t)$ and $\xi(s,t)$ denote points on the interface $\Sigma(t)$; $\theta$ denotes the angle that $\Sigma(t)$ makes with the positive $x$-axis. Define the commutator $[A,B]=AB-BA.$

We introduce the double layered potential operator and its adjoint in $\R^2$. Assume $\Omega$ is a $C^2$ domain with boundary $\Sigma$ and outer unit vector $\vec n$. The double layered potential operator $\mathcal{K}$ is defined: for scalar-valued functions $f$ on $\Sigma$,
\begin{equation}\label{eq-double-lp}
\mathcal{K}f(\xi)=\frac {-1}{\pi} \operatorname{p.v.} \int_\Sigma \frac {(\xi-\eta)\cdot \vec n(\eta)}{|\xi-\eta|^2} f(\eta)dS(\eta), \text{ for }\xi\in \Sigma.
\end{equation}where $dS$ denotes the surface measure on $\Sigma$. Let $L^2(\Sigma,dS)$ be the $L^2$ space with respect to the surface measure $dS$. We denote $\mathcal{K}^*$ be the adjoint of $\mathcal{K}$ in $L^2$, then
\begin{equation}\label{eq-adjoint-double-lp}
\mathcal{K}^*f(\xi)=\frac {1}{\pi} \operatorname{p.v.} \int_\Sigma \frac {(\xi-\eta)\cdot \vec n(\xi)}{|\xi-\eta|^2} f(\eta)dS(\eta), \text{ for }\xi\in \Sigma.
\end{equation}
We recall a theorem due to Verchota  \cite{Verchota-1984-Layer-potential} and Kenig \cite{Kenig-book-1994-Harmonics-Anal-Elliptic-BV}; see a nice discussion of this theorem in \cite{Christ:1990:CBMS-book}. We take the form stated in \cite{Sijue:1999:3d-local}.
\begin{theorem}\label{thm-Verchota}
Assume that $\Omega$ and $\Omega^c$ are unbounded, connected Lipschitz domains, and $\Sigma$ approaches plane $x_2=0$ at infinity. Then $I\pm \mathcal{K}$: $L^2(\Sigma, dS)\to L^2(\Sigma, dS)$ and their adjoints $I\pm \mathcal{K}^*$ are invertible.
\end{theorem}
Recall the definition of the Hilbert transform on $\Sigma(t)$,
\begin{equation*}
\mathfrak{H}f(\alpha,t)=\frac 1{\pi i}\operatorname{p.v.}\int \frac {f(\beta,t)z_\beta(\beta,t)}{z(\alpha,t)-z(\beta,t)}d\beta.
\end{equation*}
It is easy to see that, for a scalar-valued function $f$,
 \begin{equation}\label{eq-a4}
 \begin{split}
\mathcal{K}f&=\operatorname{Re}\bigl(\mathfrak{H}\bigr)f, \\
\mathcal{K}^* f&=-\operatorname{Re} \left( e^{i\theta}\mathfrak{H}( e^{-i\theta}f)\right)=-\operatorname{Re} \left( e^{i\theta}[\mathfrak{H}, e^{-i\theta}]f+\mathfrak{H}f\right).
\end{split}
\end{equation}

We cite a lemma in \cite{Sijue:2009:2d-almost-global} to commute the derivative with $\mathfrak{H}$.
\begin{lemma}\label{le-2}
Assume that $f\in C^1(\R\times (0,T))$ satisfying $f_\alpha\to 0$ as $|\alpha|\to \infty$. Then
\begin{align}
[\partial_t,\mathfrak{H}]f &=[z_t,\mathfrak{H}]\frac {f_\alpha}{z_\alpha},\\
\partial_\alpha \mathfrak{H}f &=z_\alpha \mathfrak{H}\frac {f_\alpha}{z_\alpha}.
\end{align}
\end{lemma}

In Section \ref{sec:pass2arc-length+energy}, we will consider periodic functions. For periodic functions, the kernels in the Hilbert transform $\mathcal{H}$ and the Cauchy transform $\mathfrak{H}$ will take different forms from the real line case;
\begin{equation}\label{eq-39}
\mathcal{H}f(\alpha)=\frac{\operatorname{p.v.}}{\pi} \int_{\mathbb{R}}\frac {f(\beta)}{\alpha-\beta} d\beta =\frac {1}{2\pi}\int_0^{2\pi} f(\beta) \cot \frac {\alpha-\beta}{2} d\beta.
\end{equation}

\begin{remark}\label{re-5}
For periodic functions, when $\Omega$ is a space-time slab,  Theorem \ref{thm-Verchota} still holds, for instance, see \cite[Lemma 6.1]{Ambrose:2003:LWP-vortex-sheet-with-surface-tension} and the references therein.
\end{remark}

\section{Reduction to a quasilinear system}\label{sec:new-reduction}
We fix an integer $r\ge 4$. Let $\kappa$ be the change of coordinates from Lagrangian $(\alpha,t)$ to arc-length $(s,t)$. We assume that  $\ka$ is a diffeomorphism.

We also fix the notions $z$ and $\xi$: points on the interface of the water surface in the Lagrangian and arc-length coordinates. We have the relations between $z$ and $\xi$:  $$z(\alpha,t)=\xi(s,t)=\xi\circ \ka (\alpha,t).$$ Then if differentiating both sides in $\alpha$ and $t$, we have
\begin{align}
\label{eq-a1} z_\alpha(\alpha,t)&=\bigl(\xi_s \kappa_\alpha\bigr)\circ \ka ,\\
\label{eq-a2} z_t(\alpha,t)&=\bigl(\kappa_t\xi_s+\xi_t\bigr)\circ \ka,
\end{align}
where $\kappa_\alpha$ is a scalar.  Then \eqref{eq-a1} yields that, in the Lagrangian coordinate,
\begin{equation}\label{eq-1}
|z_\alpha|=|\kappa_\alpha|, \text{ as } |\xi_s|=1.
\end{equation}
Since we are assuming that the interface is nonself-intersecting, i.e., $|(z(\alpha)-z(\beta))/(\alpha-\beta)|>c$ for some $c>0$ and for all $\alpha\neq \beta$. We see that $\kappa_\alpha\neq 0$. 
From \eqref{eq-a2},
\begin{equation}\label{eq-b1}
\bigl(\kappa_t\xi_s\bigr)\circ \ka =z_t(\alpha,t)-\xi_t\circ \ka ,\,\Rightarrow \kappa_t\circ \kappa^{-1}= \bigl(z_t\circ\kappa^{-1}-\xi_t\bigr)\cdot \xi_s.
\end{equation}
Here $z_1\cdot z_2$ denotes the inner product of $z_1$ and $z_2$ when regarding $z_1$, $z_2$ as vectors in $\mathbb{R}^2$. It is also that $z_1\cdot z_2 =\operatorname{Re}(\bar{z}_1z_2)$.

We set $\delta(s,t):=\ka_t \circ \ka^{-1}(s,t)$ and $u:=\delta_s$; then $\delta$ is the difference of the two tangential velocities in the Lagrangian and arc-length coordinates. In the next two subsections, we will derive a quasilinear system on $(\theta, u)$ from \eqref{eq:complex-coordinate} by differentiations and taking the real or the imaginary parts.

\subsection{An equation on $\theta_t$.} We take the time derivative to $\ln z_\alpha=\ln |z_\alpha|+i\theta$ to obtain
\begin{equation}\label{eq-3}
\frac {z_{t\alpha}}{z_\alpha} =\frac {\ka_{t\alpha}}{\ka_\alpha}+i\theta_t =(\ka_t \circ \ka^{-1})_s\circ \ka +i\theta_t=\delta_s\circ \ka+i\theta_t.
\end{equation}
After taking the complex conjugate and applying the Cauchy transform to both sides of \eqref{eq-3},
\begin{equation}\label{eq-4}
\mathfrak{H}\frac {\bar{z}_{t\alpha}}{\bar{z}_\alpha} =\mathfrak{H}(\delta_s\circ \ka)-i\mathfrak{H}(\theta_t).
\end{equation}
We write the left hand side of \eqref{eq-4}:
\begin{equation}\label{eq-5}
\begin{split}
&=\mathfrak{H}\left(\frac {\bar{z}_{t\alpha}}{z_\alpha}\frac {z_\alpha}{\bar{z}_\alpha}\right)=\frac {z_\alpha}{\bar{z}_\alpha} \mathfrak{H}\frac {\bar{z}_{t\alpha}}{z_\alpha} +[\mathfrak{H},\frac {z_\alpha}{\bar{z}_\alpha} ]\frac {\bar{z}_{t\alpha}}{z_\alpha}\\
&=\frac {\bar{z}_{t\alpha}}{\bar{z}_\alpha}+[\mathfrak{H},\frac {z_\alpha}{\bar{z}_\alpha} ]\frac {\bar{z}_{t\alpha}}{z_\alpha},
\end{split}
\end{equation} where we have used the fact, by using the second identity in Lemma \ref{le-2}
 $$\mathfrak{H}\frac {\bar{z}_{t\alpha}}{z_{\alpha}}=\mathfrak{H}\frac {\partial_\alpha\bar{z}_{t}}{z_{\alpha}}=\frac {1}{z_\alpha} \partial_\alpha \mathfrak{H}\bar {z}_t=\frac {1}{z_\alpha} \partial_\alpha \bar {z}_t=\frac {\bar{z}_{t\alpha}}{z_{\alpha}}. $$  Note that here we have used the second equation in \eqref{eq:complex-coordinate}.

Hence we see that
\begin{equation}\label{eq-6}
[\mathfrak{H},\frac {z_\alpha}{\bar{z}_\alpha} ]\frac {\bar{z}_{t\alpha}}{z_\alpha} =\mathfrak{H}(\delta_s\circ \ka)-i\mathfrak{H}(\theta_t)-\delta_s\circ \kappa+i\theta_t.
\end{equation}
Taking the imaginary part,
\begin{equation}\label{eq-7}
\begin{split}
\theta_t &=-\operatorname{Im}\mathfrak{H}(\delta_s\circ \ka)+\operatorname{Re}\mathfrak{H}(\theta_t)+\operatorname{Im}[\mathfrak{H},\frac{z_\alpha}{\bar{z}_\alpha}]\frac {\bar{z}_{t\alpha}}{z_\alpha}\\
&=-\frac {\mathfrak{H}-\overline{\mathfrak{H}}}{2i}\bigl(\delta_s\circ \ka \bigr)+\operatorname{Re}\mathfrak{H}(\theta_t)+\operatorname{Im}[\mathfrak{H},\frac{z_\alpha}{\bar{z}_\alpha}]\frac {\bar{z}_{t\alpha}}{z_\alpha}\\
&=i\mathfrak{H}\bigl(\delta_s\circ \ka \bigr)-i\frac {\mathfrak{H}+\overline{\mathfrak{H}}}{2}\bigl(\delta_s\circ \ka \bigr)+\operatorname{Re}\mathfrak{H}(\theta_t)+\operatorname{Im}[\mathfrak{H},\frac{z_\alpha}{\bar{z}_\alpha}]\frac {\bar{z}_{t\alpha}}{z_\alpha}\\
&=i\mathfrak{H}\bigl(\delta_s\circ \ka \bigr)-i\operatorname{Re}\mathfrak{H}\bigl(\delta_s\circ \ka \bigr)+\operatorname{Re}\mathfrak{H}(\theta_t)+\operatorname{Im}[\mathfrak{H},\frac{z_\alpha}{\bar{z}_\alpha}]\frac {\bar{z}_{t\alpha}}{z_\alpha}\\
&=i\mathfrak{H}\bigl(\delta_s\circ \ka \bigr) +\phi(\theta_s,\delta\circ \ka)\\
&=i\mathfrak{H}\bigl(u\circ \ka \bigr) +\phi(\theta,u\circ \ka),
\end{split}
\end{equation}where \begin{equation}\label{eq-8}
\phi(\theta,u\circ \ka)
 =-i\operatorname{Re}\mathfrak{H}\bigl(u\circ \ka \bigr)+\operatorname{Re}\mathfrak{H}(\theta_t)+\operatorname{Im}[\mathfrak{H},\frac{z_\alpha}{\bar{z}_\alpha}]\frac {\bar{z}_{t\alpha}}{z_\alpha}.
\end{equation}

\subsection{An equation on $u_t$.} We will derive an equation on $u_t$ from the Euler equation. We begin with \eqref{eq:complex-coordinate} and decompose $\nabla P$ along the tangential and normal directions:
\begin{equation}\label{eq-9}
\begin{split}
z_{tt}+i=-\nabla P& =\langle -\nabla P,  \frac {z_\alpha}{|z_\alpha|} \rangle \frac {z_\alpha}{|z_\alpha|}+\langle -\nabla P,  \frac {iz_\alpha}{|z_\alpha|} \rangle \frac {iz_\alpha}{|z_\alpha|}\\
&=-\frac {\partial_\alpha}{|z_\alpha|}P e^{i\theta} +i(-\frac {\partial P}{\partial n}) e^{i\theta}\\
&=-\frac {\partial_\alpha}{|z_\alpha|}P e^{i\theta} +ia e^{i\theta},
\end{split}
\end{equation} where
\begin{equation}\label{eq-10}
a:=-\frac {\partial P}{\partial n} \text{ denotes the Taylor sign.}
\end{equation}
Applying $\frac{\partial_\alpha}{|z_\alpha|}$ to both sides of \eqref{eq-9},
\begin{equation}\label{eq-11}
\begin{split}
\frac{\partial_\alpha}{|z_\alpha|}z_{tt}&=\frac {z_{tt\alpha}}{|z_\alpha|}=[\frac {1}{|z_\alpha|},\partial_t^2]z_\alpha +\partial_t^2\bigl(\frac {z_\alpha}{|z_\alpha|}\bigr)\\
&=-\frac{\partial_\alpha}{|z_\alpha|}\bigl( \frac{\partial_\alpha P}{|z_\alpha|}\bigr)e^{i\theta}+i\frac{\partial_\alpha a}{|z_\alpha|}e^{i\theta}-\frac{\partial_\alpha P}{|z_\alpha|}e^{i\theta} \frac {i\partial_\alpha \theta}{|z_\alpha|}+iae^{i\theta} \frac {i\partial_\alpha \theta}{|z_\alpha|}.
\end{split}
\end{equation}Recalling that $P=-\frac {\partial_\alpha \theta}{|z_\alpha|}$, and multiplying both sides of \eqref{eq-11} by $e^{-i\theta}$ and then taking real parts, we have
\begin{equation}\label{eq-12}
\operatorname{Re}\left(e^{-i\theta}[\frac {1}{|z_\alpha|},\partial^2_t]z_\alpha +e^{-i\theta}\partial_t^2\bigl(\frac {z_\alpha}{|z_\alpha|}\bigr)\right) =\bigl(\frac {\partial_\alpha}{|z_\alpha|}\bigr)^3\theta -a\frac {\partial_\alpha}{|z_\alpha|}\theta.
\end{equation}
By using the product rule $[A, B^j]f=\sum_{k=1}^j B^{j-k}[A,B]B^{k-1}f$, we see that
\begin{equation}\label{eq-13}
\begin{split}
[\frac {1}{|z_\alpha|},\partial_t^2]z_\alpha &=\partial_t\bigl([\frac {1}{|z_\alpha|},\partial_t]z_\alpha\bigr)+[\frac {1}{|z_\alpha|},\partial_t]z_{t\alpha}\\
&=\partial_t\bigl(-\partial_t(\frac {1}{|z_\alpha|})z_\alpha\bigr)-\partial_t\bigl(\frac {1}{|z_\alpha|}\bigr) z_{t\alpha}\\
&=\partial_t \bigl(\frac {1}{|z_\alpha|}\delta_s \circ \ka z_\alpha\bigr)+\frac {1}{|z_\alpha|} (\delta_s\circ \ka)z_{t\alpha}\\
&=\partial_t \bigl(\delta_s \circ \ka e^{i\theta}\bigr)+\frac {z_{t\alpha}}{|z_\alpha|} (\delta_s\circ \ka)\\
&=\partial_t(\delta_s \circ \ka)e^{i\theta}+\delta_s\circ \ka i\theta_te^{i\theta}+\frac {z_{t\alpha}}{|z_\alpha|}(\delta_s\circ \ka),
\end{split}
\end{equation} where we have used that \begin{equation*}
\begin{split}
& [\frac {1}{|z_\alpha|}, \partial_t]z_\alpha =\frac {z_{t\alpha}}{|z_\alpha|} -\partial_t(\frac {1}{|z_\alpha|}z_\alpha)=-\partial_t\bigl(\frac {1}{|z_\alpha|}\bigr)z_\alpha,\\
&\partial_t(1/|z_\alpha|)=\partial_t(1/\ka_\alpha)=-\ka_{t\alpha}/\ka_\alpha^2=-\delta_s\circ \ka/\ka_\alpha.
\end{split}
\end{equation*} On the other hand,
\begin{equation}\label{eq-14}
\partial_t^2\bigl(\frac {z_\alpha}{|z_\alpha|}\bigr)=\partial_t^2(e^{i\theta}) =(-\theta_t^2+i\theta_{tt})e^{i\theta}.
\end{equation}
Then taking the real parts of \eqref{eq-12}, and using \eqref{eq-13} and \eqref{eq-14}, we obtain
\begin{equation}\label{eq-15}
\partial_t(\delta_s \circ \ka) +\operatorname{Re} \bigl( \frac {e^{-i\theta}z_{t\alpha}}{|z_\alpha|}(\delta_s\circ \ka)\bigr)-\theta_t^2=\bigl(\frac {\partial_\alpha}{|z_\alpha|}\bigr)^3\theta -a\frac {\partial_\alpha}{|z_\alpha|}\theta.
\end{equation}
Thus we conclude that
\begin{equation}\label{eq-16}
\partial_t(\delta_s \circ \ka)=\bigl(\frac {\partial_\alpha}{|z_\alpha|}\bigr)^3\theta -a\frac {\partial_\alpha}{|z_\alpha|}\theta +\psi(\theta,\delta_s\circ \ka),
\end{equation} i.e.,
\begin{equation}\label{eq-a3}
\partial_t(u\circ \ka)=\bigl(\frac {\partial_\alpha}{|z_\alpha|}\bigr)^3\theta -a\frac {\partial_\alpha}{|z_\alpha|}\theta +\psi(\theta,u\circ \ka),
\end{equation}
where
\begin{equation}\label{eq-17}
\psi(\theta,u\circ \ka):=-\operatorname{Re} \bigl( \frac {e^{-i\theta}z_{t\alpha}}{|z_\alpha|}(u\circ \ka)\bigr)+\theta_t^2.
\end{equation}

To conclude this section, we have derived a system of equations on $(u, \theta)$, where $u=\delta_s$ and $\delta:=\ka_t\circ \ka^{-1}$, $\ka$ is the change of coordinates from Lagrangian to arc-length; and $\theta$ is the tangent angle that the interface makes with the positive $x$-axis. In the Lagrangian coordinate,
\begin{equation}\label{eq-lagrangian}
\begin{cases}
\theta_t &=i\mathfrak{H}(u\circ \ka) +\phi,\\
(u\circ \ka)_t&=\bigl(\frac {\partial_\alpha}{|z_\alpha|}\bigr)^3\theta -a\frac {\partial_\alpha}{|z_\alpha|}\theta +\psi
\end{cases}
\end{equation}where $\phi,\,\psi$ are defined in \eqref{eq-8} and \eqref{eq-17}, respectively. This completes the derivation of the system \eqref{eq-coordinate-invariant-form}, which therefore completes the proof of Proposition \ref{thm-coordinate-changes}.

\section{Equations on the Taylor sign}\label{sec:taylor-sign}
As observed in \cite{Ambrose-Masmoudi:2005:zero-surface-tension-2d-water-wave}, in order to establish \emph{a priori} energy inequality for the quasilinear system \eqref{eq-lagrangian}, the Taylor sign has to be taken into account: the explicit form of $a_s$ is used to cancel some high order term contributions in the energy inequality. Here we adopt a similar procedure to derive an explicit form of $a_s$ as in the previous subsection on $u_t$.

Recall that
\begin{equation}\label{eq-19}
z_{tt}+i =-\nabla P=-\frac {\partial_\alpha P}{|z_\alpha|} e^{i\theta}+iae^{i\theta}=-(\frac {\partial_\alpha }{|z_\alpha|})^2\theta e^{i\theta} +ia e^{i\theta}.
\end{equation}

We apply $\frac {\partial_\alpha}{|z_\alpha|}$ to both sides to obtain
\begin{equation}\label{eq-21}
\frac {\partial_\alpha }{|z_\alpha|} z_{tt}=-(\frac {\partial_\alpha }{|z_\alpha|})^2P e^{i\theta}+i\frac {\partial_\alpha a}{|z_\alpha|}e^{i\theta}-\frac {\partial_\alpha P}{|z_\alpha|} e^{i\theta} \frac {i\partial_\alpha\theta }{|z_\alpha|}+iae^{i\theta}\frac {i\partial_\alpha\theta }{|z_\alpha|}.
\end{equation}
Multiplying both sides by $e^{-i\theta}$ and then taking the imaginary part, we see that
\begin{equation}
\frac {\partial_\alpha a}{|z_\alpha|}=(\frac {\partial_\alpha }{|z_\alpha|})^2\theta \frac {\partial_\alpha\theta }{|z_\alpha|}+\operatorname{Im}\bigl( \frac {\partial_\alpha z_{tt}}{|z_\alpha|}e^{-i\theta}\bigr).
\end{equation} By a similar commutator analysis as in the previous subsection, we
obtain that
\begin{equation}\label{eq-22}
\operatorname{Im}\bigl( \frac {\partial_\alpha z_{tt}}{|z_\alpha|}e^{-i\theta}\bigr)=\theta_{tt}+\delta_s\circ \ka \theta_t+ \operatorname{Im}\bigl( \frac {\partial_\alpha z_{t}}{|z_\alpha|}e^{-i\theta}\bigr)\delta_s\circ \ka.
\end{equation}
Recall that $\theta_t=i\mathfrak{H}(\delta_s\circ \ka)+\phi(\theta,\delta_s\circ \ka)$, we see that
\begin{equation}\label{eq-23}
\begin{split}
\frac {\partial_\alpha a}{|z_\alpha|}&=i \mathfrak{H}\partial_t(\delta_s\circ \ka) +\partial_t \phi+ i[\partial_t,\mathfrak{H}](\delta_s\circ \ka)+\delta_s\circ \ka \theta_t+ \operatorname{Im}\bigl( \frac {\partial_\alpha z_{t}}{|z_\alpha|}e^{-i\theta}\bigr)\delta_s\circ \ka+\bigl(\frac {\partial_\alpha}{|z_\alpha|} \bigr)^2\theta \frac {\partial_\alpha}{|z_\alpha|}\theta\\
&=i \mathfrak{H}\partial_t(\delta_s\circ \ka) +\partial_t \phi+ i[z_t,\mathfrak{H}]\frac {(\delta_s\circ \ka)_\alpha}{z_\alpha}+\delta_s\circ \ka \theta_t+ \operatorname{Im}\bigl( \frac {\partial_\alpha z_{t}}{|z_\alpha|}e^{-i\theta}\bigr)\delta_s\circ \ka+\bigl(\frac {\partial_\alpha}{|z_\alpha|} \bigr)^2\theta \frac {\partial_\alpha}{|z_\alpha|}\theta\\
&=i \mathfrak{H}\partial_t(u\circ \ka) +\omega(\theta,u\circ \ka) ,
\end{split}
\end{equation} where we have used Lemma \ref{le-2}.
\begin{equation}\label{eq-24}
\begin{split}
\omega:&=\partial_t \phi+ i[z_t,\mathfrak{H}]\frac {(\delta_s\circ \ka)_\alpha}{z_\alpha}+\delta_s\circ \ka \theta_t+ \operatorname{Im}\bigl( \frac {\partial_\alpha z_{t}}{|z_\alpha|}e^{-i\theta}\bigr)\delta_s\circ \ka+\bigl(\frac {\partial_\alpha}{|z_\alpha|} \bigr)^2\theta \frac {\partial_\alpha}{|z_\alpha|}\theta\\
&=\partial_t \phi+ i[z_t,\mathfrak{H}]\frac {(u\circ \ka)_\alpha}{z_\alpha}+u\circ \ka \theta_t+ \operatorname{Im}\bigl( \frac {\partial_\alpha z_{t}}{|z_\alpha|}e^{-i\theta}\bigr)u\circ \ka+\bigl(\frac {\partial_\alpha}{|z_\alpha|} \bigr)^2\theta \frac {\partial_\alpha}{|z_\alpha|}\theta.
\end{split}
\end{equation}

\subsection{Expressions of $a$.} In \cite{Sijue:2009:2d-almost-global}, by using the adjoint double layered potential operator, Wu determined $a-1$ explicitly as a function of the position $z=z(\cdot, t)$ and the velocity $z_t(\cdot, t)$ of the interface. This has the advantage that one can estimate $a-1$ by terms in Energy. Furthermore, by taking
the time derivative, and commuting $\partial_t$ with the double layered potential operator up to commutators, one can also estimate $a_t$ by Energy. In this section, in the presence of surface tension, we derive similar estimates for $a-1$; then by using Lemma \ref{le-2}, we compute $a_t$.

We recall the equation $z_{tt}+i=-\nabla P=iae^{i\theta}-\frac {\partial_\alpha P}{|z_\alpha|}e^{i\theta}$.
$$ -i a e^{-i\theta}=\bar z_{tt}-i +\frac {\partial_\alpha P}{|z_\alpha|} e^{-i\theta}.$$
We apply $I-\mathfrak{H}$ to both sides,
$$ -i(I-\mathfrak{H})\left(a e^{-i\theta}\right)=(I-\mathfrak{H})\left(\bar z_{tt}-i+\frac {\partial_\alpha P}{|z_\alpha|} e^{-i\theta}\right).$$
If multiplying both sides by $n=ie^{i\theta}$ and then taking the real parts,
\begin{equation}\label{eq-a5}
(I+\mathcal{K}^*)a =\operatorname{Re}\left\{ie^{i\theta}\left( \bigl(I-\mathfrak{H}\bigr)\bigl(\bar{z}_{tt}-i+\frac {\partial_\alpha P}{|z_\alpha|} e^{-i\theta}\right)\right\}.
\end{equation}
In other words, we have
\begin{equation}\label{eq-a6}
a =(I+\mathcal{K}^*)^{-1}\operatorname{Re}\left\{ie^{i\theta}\left( \bigl(I-\mathfrak{H}\bigr)\bigl(\bar{z}_{tt}-i+\frac {\partial_\alpha P}{|z_\alpha|} e^{-i\theta}\right)\right\}.
\end{equation}
We can simplify the right hand side of \eqref{eq-a6} further by observing $(I-\mathfrak{H})\bar{z}_{tt}=[z_t,\mathfrak{H}]\frac {\bar{z}_{t\alpha}}{z_\alpha}$:
\begin{equation}\label{eq-a7}
a =(I+\mathcal{K}^*)^{-1}\operatorname{Re}\left\{ie^{i\theta}\left( [z_t,\mathfrak{H}]\frac {\bar{z}_{t\alpha}}{z_\alpha}+e^{i\theta}(I-\mathfrak{H}){\bf 1}+(I-\mathfrak{H})\bigl(\frac {\partial_\alpha P}{|z_\alpha|} e^{-i\theta}\bigr)\right)\right\}.
\end{equation} Here ${\bf 1}$ denotes the constant function. 
We also need a bound on $a_t$ in terms of Energy. From the expression in \eqref{eq-a5}, if taking time derivative, we see that
\begin{equation}\label{eq-a8}
(I+\mathcal{K}^*)a_t=-[\partial_t,K^*]a+\partial_t\left\{\text{ RHS of \eqref{eq-a5} }\right\}.
\end{equation}
Then the expression for $a_t$ follows from inverting $I+\mathcal{K}^*$.

\section{Some preliminary estimates}\label{sec:pass2arc-length+energy}
In this section, we prepare some preliminary estimates for the \emph{a priori} energy inequality, which will be defined in this section.

We define the length of this curve by \begin{equation}\label{eq-length}
L(t):=\int_0^{2\pi} |z_\alpha| d\alpha.
\end{equation}
Then under the change of coordinate $\ka$, the length is still $L(t)$ because it is invariant under coordinate changes. This fact can also be explained by computing,
$$ \int_{0}^{2\pi} |z_\alpha| d\alpha= \int_{0}^{2\pi} \ka_\alpha(t) d\alpha =\int_{0}^{L(t)}\ka_\alpha\circ \ka^{-1} \frac {ds}{\ka_\alpha\circ \ka^{-1}}=\int_0^{L(t)} ds =L(t),$$
since $d\alpha=ds/\ka_\alpha\circ \ka^{-1}$ and $|z_\alpha| \circ \ka^{-1}=\ka_\alpha\circ \ka^{-1} $.

We record the time derivative here,
\begin{equation}\label{eq-42}
\begin{split}
\frac {d L(t)}{dt}& =\int_0^{2\pi} \frac {d}{dt} \ka_\alpha d\alpha=\int_0^{L(t)} \ka_{\alpha t} \circ \ka^{-1} \frac {ds}{\ka_\alpha\circ \ka^{-1}} \\
&=\int_0^{L(t)} \delta_s \ka_{\alpha} \circ \ka^{-1} \frac {ds}{\ka_\alpha\circ \ka^{-1}}\\
&=\int_0^{L(t)} \delta_s ds,
\end{split}
\end{equation}
where we have used the fact $$\partial_s \delta=\partial_s \bigl(\ka_t\circ \ka^{-1}\bigr)=\partial_s \bigl(\ka_t(\ka^{-1}(s,t),t)\bigr)=\ka_{\alpha t}\circ \ka^{-1} \partial_s \ka^{-1} =\ka_{\alpha t}\circ \ka^{-1} \frac {1}{\ka_\alpha\circ \ka^{-1}}. $$

\subsection{Passing from Lagrangian to arc-length} Let us first recall the system \eqref{eq-lagrangian} in the Lagrangian coordinate.
\begin{equation*}
\eqref{eq-lagrangian}': \qquad \begin{cases}
\theta_t &=i\mathfrak{H}(u\circ \ka) +\phi,\\
(u\circ \ka)_t&=\bigl(\frac {\partial_\alpha}{|z_\alpha|}\bigr)^3\theta -a\frac {\partial_\alpha}{|z_\alpha|}\theta +\psi
\end{cases}
\end{equation*} where $\phi,\,\psi$ are defined in \eqref{eq-8} and \eqref{eq-17}, respectively:
$$\eqref{eq-8}', \, \eqref{eq-17}': \,
\begin{cases}
\phi&=-i\operatorname{Re}\mathfrak{H}\bigl(u\circ \ka \bigr)+\operatorname{Re}\mathfrak{H}(\theta_t)+\operatorname{Im}[\mathfrak{H},\frac{z_\alpha}{\bar{z}_\alpha}]\frac {\bar{z}_{t\alpha}}{z_\alpha},\\
\psi&=-\operatorname{Re}\bigl(  \frac {e^{-i\theta}z_{t\alpha}}{|z_\alpha|}(u\circ \kappa) \bigr)+\theta_t^2.
\end{cases}$$
We also recall the expression of $\frac {\partial_\alpha a} {|z_\alpha|}$:
\begin{equation*}
\eqref{eq-23}':\quad \frac {\partial_\alpha a} {|z_\alpha|} =i\mathfrak{H}\partial_t (u\circ \ka)+ \omega(\theta,u\circ \ka),
\end{equation*}
where $\omega$ is the error term,
\begin{equation*}
\eqref{eq-24}': \quad  \omega=\partial_t \phi+ i[z_t,\mathfrak{H}]\frac {(u\circ \ka)_\alpha}{z_\alpha}+u\circ \ka \theta_t+ \operatorname{Im}\bigl( \frac {\partial_\alpha z_{t}}{|z_\alpha|}e^{-i\theta}\bigr)u\circ \ka+\bigl(\frac {\partial_\alpha}{|z_\alpha|} \bigr)^2\theta \frac {\partial_\alpha}{|z_\alpha|}\theta.
\end{equation*}

We linearize $i\mathfrak{H}$. Recall that
\begin{equation}\label{eq-25}
i\mathfrak{H}(\alpha,t)=\frac {1}{\pi} \operatorname{p.v.} \int f(\beta,t)\frac {z_\beta(\beta,t)}{z(\alpha,t)-z(\beta,t)} d\beta.
\end{equation}
This takes the form under the change of variables,
\begin{equation}\label{eq-43}
i\mathfrak{H}(s,t)=\frac {1}{\pi} \operatorname{p.v.} \int f(\beta,t)\frac {\xi_\beta(\beta,t)}{\xi(s,t)-\xi(\beta,t)} d\beta.
\end{equation}
Since we are considering the periodic waves, $\xi(\beta+L)= \xi(\beta)+L$, 
\begin{align*}
& \frac { \operatorname{p.v.}}{\pi} \int f(\beta,t)\frac {\xi_\beta(\beta,t)}{\xi(s,t)-\xi(\beta,t)} d\beta\\
&= \frac {\operatorname{p.v.}}{L} \int_0^L f(\beta) \xi_\beta\cot \frac {\pi(\xi(s)-\xi(\beta))}{L}d\beta.
\end{align*}
Writing
$$\cot \frac {\pi(\xi(s)-\xi(\beta))}{L} = \frac {1}{\xi_\beta} \cot\frac {\pi(s-\beta)}{L} +\left(\cot \frac {\pi(\xi(s)-\xi(\beta))}{L}- \frac {1}{\xi_\beta} \cot\frac {\pi(s-\beta)}{L}\right).$$
We see that 
\begin{equation}\label{eq-x1}
\begin{split}
i\mathfrak{H}(s,t) &= \frac {\operatorname{p.v.}}{L} \int_0^L f(\beta) \cot \frac {\pi(s-\beta)}{L} d\beta\\
&\quad + \frac { \operatorname{p.v.}}{L} \int_0^L f(\beta) \xi_\beta  \left(\cot \frac {\pi(\xi(s)-\xi(\beta))}{L}- \frac {1}{\xi_\beta} \cot\frac {\pi(s-\beta)}{L} \right)d\beta. 
\end{split}
\end{equation}
The first term is the Hilbert transform $\H(f)$; the second term is an error term denoted by $R(f)$.

Now we perform the change from the Lagrangian coordinate to the arc-length coordinate, and linearize $i\mathfrak{H}$ to the standard Hilbert transform $\H$ for the main terms in the system. Under the change of coordinates from Lagrangian to arc-length, the derivatives will change accordingly,
$$ \partial_t\mapsto \partial_t +\delta\partial_s, \, \frac {\partial_\alpha}{|z_\alpha|}\mapsto \partial_s.$$
So the system \eqref{eq-lagrangian} (or $\eqref{eq-lagrangian}'$ ) in the arc-length coordinate takes the following form:
\begin{equation}\label{eq-arclength}
\begin{cases}
\partial_t \theta &=\H(u) -\delta \partial_s \theta+ \tilde\phi (\theta, u),\\
\partial_t u&=\partial_s^3 \theta -a\partial_s \theta -\delta \partial_s u +\tilde\psi(\theta, u),
\end{cases}
\end{equation}
where $\tilde\phi$ and $\tilde\psi$ are error terms,
\begin{equation}\label{eq-36}
\begin{cases}
\tilde\phi(\theta,u)& =-i\operatorname{Re}\mathfrak{H}(u)+\operatorname{Re}\mathfrak{H}\bigl(\theta_t\circ \ka^{-1}\bigr)-\operatorname{Im}[\mathfrak{H},e^{2i\theta}]\frac {\partial_s(\bar{z}_t\circ \ka^{-1})}{\xi_s}+\mathcal{R}(u),\\
\tilde\psi(\theta,u)& =-u\operatorname{Re}\bigl(e^{-i\theta}\partial_s(z_t\circ \ka^{-1})\bigr)+\bigl(\theta_t\circ \ka^{-1}\bigr)^2.
\end{cases}
\end{equation}
Similarly,
\begin{equation}\label{eq-derivative-of-Taylor-sign}
a_s =\H(u_t)+ \tilde\omega(\theta,u),
\end{equation}
where $\tilde\omega$ is an error term,
\begin{equation}\label{eq-38}
\begin{split}
\tilde\omega &=\phi_t\circ \ka^{-1}+i[\bar{z}_t\circ \ka^{-1},\mathfrak{H}]\frac {\partial_su}{\xi_s}+u\theta_t\circ \ka^{-1}+\operatorname{Im}\bigl(e^{-i\theta}\partial_s(\bar{z}_t\circ \ka^{-1})\bigr)u\\
&\qquad +\partial_s^2\theta\partial_s\theta +\H(\delta \partial_s u) +\mathcal{R}(\partial_t u+\delta \partial_s u). 
\end{split}
\end{equation}

\subsection{Estimates on the error terms} We first observe that for a scalar-valued function $f$,
\begin{equation}\label{eq-a12}
\begin{split}
\operatorname{Re}\mathfrak{H}(f)&=\frac {\operatorname{p.v.}}{2\pi i} \int f(\beta,t) \left( \frac {\xi_\beta}{\xi(s,t)-\xi(\beta,t)}-\frac {\bar{\xi}_\beta}{\bar{\xi}(s,t)-\bar{\xi}(\beta,t)}\right) d\beta \\
&=\operatorname{Im}  \frac {1}{\pi } \int f(\beta,t) \frac {\xi_\beta}{\xi(s)-\xi(\beta)} d\beta \\
&= \operatorname{Im} \frac 1L \int_0^L f(\beta, t) \partial \xi(\beta) \cot \frac {\pi(\xi(s)-\xi(\beta))}{L}d\beta \\
&=\operatorname{Im} \frac 1L \int_0^L f(\beta, t) \left(  \partial \xi(\beta) \cot \frac {\pi(\xi(s)-\xi(\beta))}{L}-\cot \frac {\pi(s-\beta)}{L}\right) d\beta \\
&=\operatorname{Im}\mathcal{R}(f).
\end{split}
\end{equation}
So an estimate on $\operatorname{Re} \mathfrak{H} (f)$ reduces to estimating $\operatorname{Im}\mathcal{R}(f)$ for scalar functions.

Let $r\ge 4$ be an integer. We define the Energy $\E(t)$ for the system \eqref{eq-arclength} in the arc-length coordinate: for $(\theta, \delta, \gamma)\in H^{r+1}\times H^{r+1/2}\times H^r$,
where $\gamma$ is the vortex sheet density on $\Sigma(t)$, see Lemma \ref{le-biot-savart},
\begin{equation}\label{eq-energy}
\begin{split}
\E(t)&=\|\theta\|_2^2+\|\delta\|_{L^2}^2+\|\gamma\|_{L^2}^2+\|u\|_2^2+L^2(t)+\sum_{k=1}^{r-1}\langle \partial_s^{k}\gamma, \partial_s^{k}\gamma \rangle + \sum_{k=1}^r \bigl(\E_k^1+\E^2_k+\E^3_k\bigr),
\end{split}
\end{equation}
where $u=\delta_s$,  $\langle f,\,g \rangle=\int f\bar{g} ds $ and $D=\H\partial_s$,
\begin{equation}\label{eq-energy-subterms}
\begin{cases}
\E_k^1(t) & := \frac 12 \left(\langle \partial_s^{k+1} \theta, \partial_s^{k+1} \theta  \rangle +
\langle a \partial_s^{k} \theta, \partial_s^{k} \theta  \rangle + \langle \partial_s^{k-1} u, D\partial_s^{k-1} u \rangle\right),\\
\E_k^2(t)&:= \langle \partial_s^{k-1} u, (10 \|\theta_s(0)\|_{L^\infty_{s}} -\theta_s)\partial_s^{k-1} u \rangle,\\
\E^3_k(t)&:= 10 \|\theta_s(0)\|_\infty \langle \partial_s^{k} \theta, D\partial_s^k \theta  \rangle.
\end{cases}
\end{equation}

In what follows, we will see that the error terms in $\theta_t, u_t$ and $a_s$ are bounded by $\E(t)$. In view of \eqref{eq-36} and \eqref{eq-38}, we first study the following three operators:
\begin{align}
\label{eq-a35}
\mathcal{R}(f)&= \frac {\operatorname{p.v.}}{L} \int_0^L f(\beta) \xi_\beta \left(\cot \frac {\pi(\xi(s)-\xi(\beta))}{L}- \frac {1}{\xi_\beta} \cot\frac {\pi(s-\beta)}{L} \right)d\beta,\\
\label{eq-a36} [\mathfrak{H}, e^{2i\theta(s)}] \frac {f_s}{\xi_s} &=-\frac {\operatorname{p.v.}}{ iL } \int_0^L f_\beta(\beta)
\bigl( e^{2i\theta(s)}-e^{2i\theta(\beta)}\bigr)  \cot \frac {\pi (\xi(s)-\xi(\beta))}{L}d\beta,\\
\label{eq-a37} [z_t\circ\ka^{-1}, \mathfrak{H}] \frac {f_s}{\xi_s}&=\frac {\operatorname{p.v.}}{iL } \int_0^L \frac { f_\beta(\beta)}{\xi_\beta}
\bigl( z_t\circ\ka^{-1}(s)-z_t\circ\ka^{-1}(\beta)\bigr) \cot \frac {\pi(s-\beta)}{L}d\beta.
\end{align}
Recall that $C(\E)$ is a polynomial in $\E(t)$, which may change from line to line. We recall a lemma \cite[Lemma 3.9]{Ambrose:2003:LWP-vortex-sheet-with-surface-tension}.

\begin{lemma}\label{le-algebra}
Let $r>\frac 12$ and $r\ge r'\ge 0$. If $f\in H^r$ and $g\in H^{r'}$, then $fg\in H^{r'}$ with the estimate
$$ \|fg\|_{H^{r'}} \le C \|f\|_{H^r} \|g\|_{H^{r'}}.$$
\end{lemma}

To study \eqref{eq-a35},
\begin{lemma}\label{le-bound-on-R}
Let $\theta\in H^{r+1}$ and $f\in H^{r-1/2}$ with $r\ge 4$. Then for $0\le k\le r+1$,
\begin{equation}\label{eq-a33}
\|\mathcal{R}(f)\|_{H^k} \le C(\E).
\end{equation}
\end{lemma}

To study the operators in \eqref{eq-a36} and \eqref{eq-a37}, we have
\begin{lemma}\label{le-more-error-bounds}
Let $f=\bar{z}_t\circ \ka^{-1}\in H^{r+\frac 12}$.
\begin{itemize}
\item[1.] For $0\le k\le r+1$,
\begin{equation}\label{eq-a38}
\|[\mathfrak{H}, e^{2i\theta(s)}] \frac {f_s}{\xi_s}\|_{H^k} \le C(\E).
\end{equation}

\item[2.] For $0\le k\le r$,
\begin{equation}\label{eq-a39}
\|[f, \mathfrak{H}] \frac {\partial_s u}{\xi_s}\|_{H^{k}} \le C(\E).
\end{equation}
\end{itemize}
\end{lemma}

The proofs of these lemmas are similar to that of \cite[Lemma 3.5]{Ambrose:2003:LWP-vortex-sheet-with-surface-tension}. We omit the proofs.

By using Theorem \ref{thm-Verchota}, Lemma \ref{le-bound-on-R} and \eqref{eq-a38} in Lemma \ref{le-more-error-bounds}, we conclude that the error terms in $\theta_t$ and $u_t$ are bounded by the Energy. To get the $H^{r-\frac 12}$ boundedness of $\theta_t\circ \kappa^{-1}$, we need to invoke the boundedness of the operator $ I-\mathcal{K}$; the proof of this fact is similar to that of \cite[Lemma 6.4, d.]{Sijue:1999:3d-local}
\begin{proposition}\label{prop-error-theta-t} For $0\le k\le r+1$,
$$\|\tilde\phi\|_{H^k} \le C(\E).$$
\end{proposition}

\begin{proposition}\label{prop-error-u-t} For $0\le k\le r-1$,
$$\|\tilde\psi\|_{H^{k+1/2}}\le C(\E).$$
\end{proposition}

Finally the error term in $a_s$ is bounded by the Energy, too. 
\begin{proposition}\label{prop-error-a-s} For $0\le k\le r-2$,
$$\|\tilde\omega\|_{H^{k+1/2}}\le C(\E).$$
\end{proposition}

This proposition is contained in \cite[Proposition 2.4]{Ambrose-Masmoudi:2005:zero-surface-tension-2d-water-wave}. For the convenience of reader, we give the proof below. 
\begin{proof}
We recall the expression of $\tilde\omega$ in \eqref{eq-38}. By using Lemma \ref{le-algebra} and \eqref{eq-a39} in Lemma \ref{le-more-error-bounds}, we see that the terms except for $\phi_t\circ \ka^{-1}$ are bounded by $C(\E)$; so we focus on proving
$$ \|\phi_t\circ \ka^{-1}\|_{H^{r-3/2}} \le C(\E),$$
where, for $\phi_t$ defined in \eqref{eq-8} in the Lagrangian coordinate, $\phi_t\circ \ka^{-1}$ takes the following form
\begin{equation}\label{eq-45}
\phi_t\circ \ka^{-1} =(\partial_t +\delta \partial_s) \left(-i\operatorname{Re}\mathfrak{H}(u)+\operatorname{Re}\mathfrak{H}(\theta_t\circ \ka^{-1})+\operatorname{Im}[\mathfrak{H},e^{2i\theta}]\frac {\partial_s(\bar{z}_{t}\circ \ka^{-1})}{\xi_s}\right).
\end{equation}
The derivative $\delta \partial_s$ acting on the terms in the bracket above is bounded by using Lemma \ref{le-algebra} and the fact $u\in H^{r-1/2}$. We investigate the time derivative $\partial_t$ term by term. Firstly we recall that $\operatorname{Re}\mathfrak{H}(f) =\operatorname{Im}\mathcal{R}(f)$ for a real-valued function;
$$ R(u) =\frac {p.v.}{L}\int_0^{L} u\xi_\beta \left( \cot \frac {\pi(\xi(s)-\xi(\beta))}{L}-\frac {1}{\xi_\beta}\cot \frac {\pi(s-\beta)}{L}\right)  d\beta.$$
Applying $\partial_t$, we see that one term is
$$\frac {p.v.}{L}\int_0^{L} u_t\xi_\beta \left( \cot \frac {\pi(\xi(s)-\xi(\beta))}{L}-\frac {1}{\xi_\beta}\cot \frac {\pi(s-\beta)}{L}\right)  d\beta. $$
This term is bounded by $C(\E)$. The estimates on the second term and the third in \eqref{eq-45} can be done similarly. So the proof of Lemma \ref{prop-error-a-s} is complete.
\end{proof}

We recall two lemmas from Ambrose \cite{Ambrose:2003:LWP-vortex-sheet-with-surface-tension}. These will be used in Proposition \ref{prop-error-commutators}. 

\begin{lemma}\cite[Lemma 3.7]{Ambrose:2003:LWP-vortex-sheet-with-surface-tension}\label{le-commutators-1}
For $\psi \in H^r$, the operator $[\mathcal{H},\psi]$ is bounded from $H^0$ to $H^{r-1}$. Also $[\H,\psi]$ is bounded from $H^{-1}$ to $H^{r-2}$. For $i=0$ or $i=-1$, we have 
$$\|[\H,\psi]f\|_{H^{r-1+i}} \le C \|f\|_{H^{i}}\|\psi\|_{H^r}. $$
\end{lemma}

\begin{lemma}\cite[Corollary 3.8]{Ambrose:2003:LWP-vortex-sheet-with-surface-tension}\label{le-commutators-2}
For $r\ge 3$ and $\psi \in H^r$, the operator $[\H,\psi]$ is bounded from $H^{r-2}$ to $H^{r}$. For $r\ge 4$ and $\psi \in H^{r-\frac 12}$, the operator $[\H,\psi]$ is bounded from $H^{r-2}$ to $H^{r-\frac 12}$.  For $i=0$ or $i=-\frac 12$, we have 
$$\|[\H,\psi]f\|_{H^{r+i}} \le C \|f\|_{H^{r-2}}\|\psi\|_{H^{r+i}}. $$
\end{lemma}

We conclude this section with a commutator estimate, which is useful in establishing the energy inequality in Section \ref{sec: energy}. 

\begin{proposition}\label{prop-error-commutators}
For $\delta\in H^{r+1/2}$ and $f\in H^{r-1/2}$ for an integer $r\ge 4$. Then
\begin{equation}\label{eq-52}
\|[\H,\delta] \partial^{r} f\|_{H^{1/2}} \le C \|\delta\|_{H^{r+1/2}} \|f\|_{H^{r-1/2}}.
\end{equation}
\end{proposition}
\begin{proof}
We set
$$ T(g)=[\H,\delta]g =\int \frac {\delta(s)-\delta(y)}{s-y} g(y)dy =:\int K(s,y)g(y)dy,$$
where $K(s,y):=\frac {\delta(s)-\delta(y)}{s-y}$.
Then
\begin{equation}\label{eq-53}
[\H,\delta] \partial^r u = [T,\partial^r] u+\partial^r T(u)=\sum_{j=1}^r \partial^{r-j}[T,\partial] \partial^{j-1} u+\partial^r T(u).
\end{equation}
Firstly, by Lemma \ref{le-commutators-2},
\begin{equation}\label{eq-54}
\|\partial^{r+1/2} T(u) \|_{L^2} \le \|[\H,\delta] u\|_{H^{r+1/2}} \le C \|\delta\|_{H^{r+1/2}} \|u\|_{H^{r-1}}.
\end{equation}
Then for the first summation in \eqref{eq-53}, we only consider $j=1$ and $j=r$ as the other cases are similar.

For $j=1$,
\begin{equation}\label{eq-55}
\| \partial^{1/2} \partial^{r-1}[T,\partial]u \|_{L^2} \le \|[T,\partial]u \|_{H^{r-1/2}} = \|[\H,\partial \delta]u \|_{H^{r-1/2}}\le \|\partial \delta \|_{H^{r-1/2}} \|u\|_{H^{r-2}},
\end{equation}
by using $\partial \delta \in H^{r-1/2}$ and Lemma \ref{le-commutators-2}.

For $j=r$, we consider
\begin{equation}\label{eq-56}
\|[T,\partial]\partial^{r-1} u\|_{H^{1/2}} = \|\H(\partial \delta \partial^{r-1} u) -\partial \delta \H(\partial^{r-1}u)\|_{H^{1/2}} \le C \|\delta\|_{H^{r+1/2}} \|u\|_{H^{r-1/2}}
\end{equation}
by Lemma \ref{le-algebra}. So the proof of Proposition \ref{prop-error-commutators} is complete.
\end{proof}

\section{The \emph{a priori} energy}\label{sec: energy}
We recall the expression of $\E$ in the previous section: for $(\theta, \delta, \gamma)\in H^{r+1}\times H^{r+1/2}\times H^r$,
\begin{equation*}
\eqref{eq-energy}':
\begin{split}
\E(t)&=\|\theta\|_2^2+\|\delta\|_{L^2}^2+\|\gamma\|_{L^2}^2+\|u\|_2^2+L(t)+\sum_{k=1}^{r-1}\langle \partial_s^{k}\gamma, \partial_s^{k}\gamma \rangle + \sum_{k=1}^r \bigl(\E_k^1+\E^2_k+\E^3_k\bigr),
\end{split}
\end{equation*}
where $u=\delta_s$,  $\langle f,\,g \rangle=\int f\bar{g} ds $ and $D=\H\partial_s$,
\begin{equation*}
\eqref{eq-energy-subterms}': \quad
\begin{cases}
\E_k^1(t) & := \frac 12 \left(\langle \partial_s^{k+1} \theta, \partial_s^{k+1} \theta  \rangle +
\langle a \partial_s^{k} \theta, \partial_s^{k} \theta  \rangle + \langle \partial_s^{k-1} u, D\partial_s^{k-1} u \rangle\right),\\
\E_k^2(t)&:= \langle \partial_s^{k-1} u, (10 \|\theta_s(0)\|_{L^\infty_{s}} -\theta_s)\partial_s^{k-1} u \rangle,\\
\E^3_k(t)&:= 10 \|\theta_s(0)\|_\infty \langle \partial_s^{k} \theta, D\partial_s^k \theta  \rangle.
\end{cases}
\end{equation*}
We make the following remarks.
\begin{itemize}
\item[1.] The quantity $\gamma$ denotes the vortex density on the interface, see the relation between the normal velocity of the water and $\gamma$ in Section \ref{sec:serveal-identities} through the Birkhoff-Rott integral and the Biot-Savart law. Note that we assume $\gamma \in H^r$ but only include $\|\gamma\|^2_{H^k}$ up to $r-1$ derivatives in the energy $\E$. This is because $\gamma \in H^r$ can be recovered from the information $\gamma\in H^{r-1}$, $\delta\in H^{r+1/2}$, see the proof of Theorem \ref{thm-energy-inequality}.

\item[2.] We note that $\E^2_k(t)$ includes a factor $10 \|\theta_s(0)\|_{L^\infty_s} -\theta_s$, which is designed to cancel some higher order terms coming out of ${d\E^1_k}/{dt}$. Here $0$ denotes the time zero. If we assume that $\E(t) \le M$ for some $M>0$, then it is nonnegative over a period of time only depending on $M$, which follows from the Sobolev embedding in the time variable that $\partial_t \theta_s$ can be controlled by $\E$ and hence by $M$.
\end{itemize}

\begin{proof}[Proof of Theorem \ref{thm-energy-inequality}] We first record several time derivatives:
\begin{equation}\label{eq-57}
\begin{cases}
\theta_t &= \mathcal{H}(u)-\delta \partial_s \theta +\tilde{\phi},\\
u_t & = \partial_s^3\theta-a\partial_s \theta-\delta \partial_s u+\tilde{\psi},\\
\delta_t &=\theta_{ss}-\delta \delta_s -\sin \theta+U\theta_t-T_t.\\
(I-K^*)(\gamma_t/2) &=[\partial_t, K^*](\gamma/2)+\theta_{ss}-\delta \delta_s -\sin \theta+U\theta_t,\\
\partial_t L(t)&=\int_0^{L(t)} \delta_s ds.
\end{cases}
\end{equation}
Here from Section \ref{sec:serveal-identities}, $T$ denotes the tangential velocity of $\xi_t$ in the arc-length variable and satisfies $T_s= U\theta_s$, and $U$ is the normal velocity defined in \eqref{eq-normal-velocity},
$$ U= z_t \cdot \vec n= \operatorname{Re} \left(\bar{z}_tie^{i\theta}\right) =
\operatorname{Re} \left(\frac {e^{i\theta}}{2\pi } \int \frac {\gamma}{\xi(s)-\xi(\beta)}d\beta \right). $$

Next we establish the claim in Theorem \ref{thm-energy-inequality} in the following steps.

\textbf{Step 1.} By using the equations in \eqref{eq-57}, it is not hard to see that the time derivatives of the first five terms in $\E$ are bounded by $C(\E)$.  For instance,
$$ \langle u,\,u_t\rangle =\langle u,\,\partial_s^3 \theta-a\partial_s\theta-\delta \partial_s u+\tilde{\psi}\rangle
=-\langle u,\, a\partial_s \theta\rangle -\langle u,\, \delta\partial_s u\rangle -C(\E).  $$
The second term above is easier as $\langle u,\, \delta\partial_s u\rangle =\frac 12 \int \delta \partial_s u^2 =-\frac 12 \int u^3 =C(\E)$ by the Sobolev embedding. For the first term, we have
$$\langle u,\, a\partial_s \theta\rangle \le C(\E). $$

\textbf{Step 2.} We show that if $\delta\in H^{r+1/2}$, $\xi_t\in L^2$ and $\gamma\in H^{r-1}$, then $\gamma\in H^r$. Recall that $\delta =(I-K^*)(\gamma/2)-T$, then $\gamma \in L^2$ follows easily.  By induction, we consider the $r$-derivative of $\gamma$. Since $T_s=U\theta_s$,
\begin{equation}\label{eq-65}
(I-K^*)(\partial^r \gamma/2)=\partial^r \delta +\partial^{r-1}(U\theta_s)+[\partial^r,K^*](\gamma/2).
\end{equation}
Since $[K^*, \partial^r]=\sum_{k=1}^r\partial^{r-k}[\partial, K^*]\partial^{k-1}$,  \eqref{eq-normal-velocity} in Section \ref{sec:serveal-identities}, then the right hand side of \eqref{eq-65} will be in $L^2$ as $(\delta, \gamma, \theta) \in H^{r+1/2}\times H^{r-1}\times H^{r+1}$.  So $\partial^r \gamma \in L^2$. This implies that $\gamma \in H^r$ and is bounded by $\E$ if the terms in the definition of $\E$ are bounded.

Then we show that $\sum_{k=1}^{k-1} \langle\partial^k \gamma, \partial^k \gamma_t\rangle \le C(\E).$ Recall the expression of $\gamma_t$ in \eqref{eq-57}, it is bounded by $\E$.

\textbf{Step 3.} We investigate the time derivatives of $\E_k^i (t)$ for each $i=1,2,3$ and given $1\le k\le r$. We first claim,
\begin{equation}\label{eq-31}
\frac {d \E_k^1 (t)}{dt}=-2\langle \partial_s^{k}u, \,\theta_s \partial_s^{k+1} \theta \rangle +C(\E).
\end{equation} Indeed,
\begin{equation}\label{eq-72}
\begin{split}
\frac {d \E_k^1 (t)}{dt} & =\left(\langle \partial_s^{k+1} \theta, \partial_s^{k+1} \theta_t  \rangle +
\langle a\partial_s^{k} \theta, \partial_s^{k} \theta_t  \rangle\right) + \langle \partial_s^{k-1} u, D\partial_s^{k-1} u_t \rangle+C(\E)\\
&=:I+II+C(\E),
\end{split}
\end{equation}
where we have used the bound on $a_t$. By using the equation for $\theta_t=-\delta \theta_s+ H(u)+\tilde\phi$,
\begin{equation}\label{eq-68}
\begin{split}
I& =\langle\partial_s^{k+1}\theta, \partial_s^{k+1}(-\delta \theta_s) \rangle +\langle\partial_s^{k+1}\theta, \partial_s^{k+1}H(u) \rangle\\
&\qquad +\langle a\partial_s^{k}\theta, \partial_s^{k}(-\delta \theta_s) \rangle +\langle a\partial_s^{k}\theta, \partial_s^{k}H(u) \rangle+C(\E)\\
&=\langle\partial_s^{k+1}\theta, \partial_s^{k+1}(-\delta \theta_s) \rangle +\langle\partial_s^{k+1}\theta, \partial_s^{k+1}H(u) \rangle +C(\E),
\end{split}
\end{equation} by using the bound on $\phi$ and $a$. For the first term in \eqref{eq-68},
\begin{equation}\label{eq-69}
\begin{split}
\langle\partial_s^{k+1}\theta, \partial_s^{k+1}(-\delta \theta_s) \rangle
&=-\langle\partial_s^{k+1}\theta, \partial_s^k u \theta_s \rangle- \langle\partial_s^{k+1}\theta, \delta \partial_s^{k+2}\theta \rangle+C(\E)\\
&=-\langle\partial_s^{k+1}\theta, \partial_s^k u \theta_s \rangle+C(\E)
\end{split}
\end{equation} because,  by integration by parts,
\begin{equation}\label{eq-70}
\langle\partial_s^{k+1}\theta, \delta \partial_s^{k+2}\theta \rangle=-\langle\partial_s^{k+2}\theta, \delta \partial_s^{k+1}\theta \rangle+C(\E).
\end{equation}Hence we rewrite \eqref{eq-68} as
\begin{equation}\label{eq-71}
I=-\langle\partial_s^{k+1}\theta, \partial_s^k u \theta_s \rangle+\langle \partial_s^{k+1}\theta, D \partial_s^{k}u\rangle +C(\E).
\end{equation}
We continue $II$ in \eqref{eq-72}. Recalling that $u_t=\theta_{sss}-a \theta_s-\delta u_s+\tilde\psi$,
\begin{equation}\label{eq-73}
\begin{split}
II&=\langle \partial_s^{k-1} u, D\partial_s^{k-1}\bigl( \theta_{sss}-a \theta_s-\delta u_s \bigr) \rangle+C(\E)\\
& =\langle \partial_s^{k-1} u, D\partial_s^{k+2}\theta \rangle -\langle \partial_s^{k-1} u,D\partial_s^{k-1}(a \theta_s)\rangle -\langle \partial_s^{k-1} u, D\partial_s^{k-1}(\delta u_s) \rangle+C(\E)\\
&=-\langle \partial_s^{k+1}\theta, D \partial_s^{k}u\rangle- \langle \partial_s^{k-1} u,D\partial_s^{k-1}(a \theta_s)\rangle+C(\E),
\end{split}
\end{equation} because the third term in the second line of \eqref{eq-73} equals
 \begin{equation}\label{eq-b21}
\begin{split}
-\langle \partial_s^{k-1} u, D\partial_s^{k-1}(\delta u_s) \rangle&=\langle H\partial_s^{k-1}u,\delta \partial_s^{k+1}u \rangle+C(\E)\\
&=-\langle H\partial_s^k u, \delta \partial_s^k u\rangle -\langle H\partial_s^{k-1} u, \delta_s \partial_s^k u \rangle +C(\E) \\
&=\langle \partial_s^k u, [H,\delta] \partial_s^k u\rangle +\langle \partial_s^k u, \delta H \partial_s^k u \rangle +C(\E)
\end{split}
\end{equation}
since $|\langle H\partial_s^{k-1} u, \delta_s \partial_s^k u \rangle| \le C(\E)$ by Lemma \ref{le-algebra}.  The equation $-\langle H\partial_s^k u, \delta \partial_s^k u\rangle =\langle \partial_s^k u, [H,\delta] \partial_s^k u\rangle +\langle \partial_s^k u, \delta H \partial_s^k u \rangle $ implies that 
\begin{equation}\label{eq-b22}
\langle \partial_s^{k-1} u, D\partial_s^{k-1}(\delta u_s) \rangle =-\frac 12\langle \partial_s^k u, [H,\delta] \partial_s^k u\rangle,
\end{equation} which is bounded by $C(\E)$ by Proposition \ref{prop-error-commutators}. 

For the second term in the third line of \eqref{eq-73}, it equals
\begin{equation}
\begin{split}
&\langle \partial_s^{k-1} u,H (\partial_s^k a \theta_s+ a \partial_s^{k+1}\theta)\rangle+C(\E)\\
&=\langle \partial_s^{k-1} u,H (\partial_s^{k-1} a_s \theta_s)\rangle+C(\E)\\
&=\langle \partial_s^{k-1} u,H (\partial_s^{k-1} \bigl( H(u_t)+\tilde\omega \bigr)\theta_s)\rangle+C(\E)\\
&=-\langle \partial_s^{k-1} u, \partial_s^{k+2}\theta\theta_s \rangle+C(\E) \\
&=\langle \partial_s^k u, \partial_s^{k+1}\theta\theta_s \rangle+C(\E)
\end{split}
\end{equation}where we have used the bound on $\omega$. Thus we see that
\begin{equation}\label{eq-74}
\frac {d\E^1(t)}{dt} =I+II+C(\E)=- 2\langle \partial_s^k u, \theta_s \partial_s^{k+1}\theta \rangle+C(\E).
\end{equation}

Next we consider $\frac {d \E_k^2 (t)}{dt}$  and $\frac {d \E_k^3 (t)}{dt}$, which are easier to handle.
\begin{equation}\label{eq-32}
\begin{split}
\frac {d \E_k^2 (t)}{dt}&=2\langle \partial_s^{k-1}u, \,(10 \|\theta_s(0)\|_{L^\infty_{s}} -\theta_s)\partial_s^{k+2}\theta \rangle +C(\E)\\
&=-2\langle \partial_s^{k}u, \,(10 \|\theta_s(0)\|_{L^\infty_s} -\theta_s)\partial_s^{k+1}\theta\rangle +C(\E).
\end{split}
\end{equation}
Thus
\begin{equation}\label{eq-33}
\frac {d \E_k^1 (t)}{dt}+\frac {d \E_k^2 (t)}{dt}=20 \|\theta_s(0)\|_{L^\infty_s} \langle \partial_s^{k-1}u, \partial_s^{k+2} \theta\rangle  +C(\E).
\end{equation}
We compute
\begin{equation}\label{eq-34}
\begin{split}
\frac {d \E_k^3 (t)}{dt}&=20\|\theta_s(0)\|_{L^\infty_s} \langle D\partial_s^k \theta_t, \partial_s^k\theta \rangle\\
&=-20\|\theta_s(0)\|_{L^\infty_s}\langle \partial_s^{k-1}u, \partial_s^{k+2}\theta \rangle +C(\E).
\end{split}
\end{equation}
To conclude,
\begin{equation}\label{eq-35}
\frac {d \E_k^1 (t)}{dt}+\frac {d \E_k^2 (t)}{dt}+\frac {d \E_k^3 (t)}{dt}=C(\E).
\end{equation}
Hence the proof of Theorem \ref{thm-energy-inequality} is complete.
\end{proof}

\section{Several useful identities}\label{sec:serveal-identities}
In this section we list several useful identities on the relation about $\gamma$, $\delta$, $\theta$ and $z_t\circ \ka^{-1}$, from which we can deduce useful estimates on one quantity in terms of others. Then we discuss the relation of $\xi_t$ with the other quantities in the energy $\E$.

We introduce the following notation: if $\xi$ denotes the points on the interface $\Sigma(t)$ in the arc length coordinate, then
\begin{equation}\label{eq-b12}
\xi_t =U \vec n+T \vec t, \text{ where }\vec t =e^{i\theta}, \, \vec n =i\vec t =ie^{i\theta},
\end{equation}
where $T=\operatorname{Re} (\xi_t e^{-i\theta})=\xi_t \cdot e^{i\theta}$ denotes the tangential velocity of the interface in the arc length coordinate, and $U$ is the normal velocity.
\begin{lemma}\label{le-biot-savart}
Let $\gamma(\beta,t)$ denotes the vorticity density on $\Sigma(t)$. Then

\begin{equation}\label{eq-biot-savt}
 \bar{z}_t\circ \ka^{-1} (s,t)  =\frac {\operatorname{p.v.}}{2\pi i} \int \frac {\gamma(\beta, t)}{\xi(s,t)-\xi(\beta,t)} d\beta+ \frac {\gamma}{2} e^{-i\theta(s,t)},
\end{equation}

\begin{equation}\label{eq-tang-velocty}
\operatorname{Re}(\bar{z}_t\circ \ka^{-1} e^{i\theta}) =(I-\mathcal{K}^*)(\frac \gamma 2),
\end{equation}

\begin{equation}\label{eq-aa25}
\delta =(z_t\circ\ka^{-1}-\xi_t)e^{-i\theta}=\operatorname{Re} (\bar{z}_t\circ \ka^{-1}e^{i\theta})-T=(I-\mathcal{K}^*)\bigl(\frac {\gamma}{2} \bigr) -T,
\end{equation}

\begin{equation}\label{eq-normal-velocity}
U= z_t \cdot \vec n= \operatorname{Re} \left(\bar{z}_tie^{i\theta}\right) =
\operatorname{Re} \left(\frac {e^{i\theta}}{2\pi } \int \frac {\gamma}{\xi(s)-\xi(\beta)}d\beta \right),
\end{equation}
and
\begin{equation}\label{eq-tangential-derivative}
T_s=\theta_s U.
\end{equation}
\end{lemma}
\begin{proof}
From the Biot-Savart law, for any $\xi\in \Omega(t)$,
\begin{equation}\label{eq-a18}
\bar{\nu}(\xi,t) =\frac {\operatorname{p.v.}}{2\pi i} \int \frac {\gamma}{\xi-\xi(\beta)} d\beta,
\end{equation}  where $\beta$ denotes the arc-length variable for $\Sigma(t)$, and $\nu$ denotes the Lagrangian velocity of the fluid in $\Omega(t)$. Then if letting $\xi\to \xi(s)$ nontangentially, then we see that
\begin{equation}\label{eq-a19}
\bar{z}_t\circ \ka^{-1}(s) =\frac {\operatorname{p.v.}}{2\pi i} \int \frac {\gamma}{\xi(s)-\xi(\beta)}d\beta +\frac {\gamma}{2} e^{-i\theta}.
\end{equation} Hence \eqref{eq-biot-savt} follows.

Multiplying both sides by $e^{i\theta}$ and then taking the real parts, we see that, from the definition of double layered potential operator,
 \begin{equation}\label{eq-a20}
\operatorname{Re} \left(\bar{z}_t\circ \ka^{-1}e^{i\theta}\right) = -\operatorname{Re}\left( \frac {\operatorname{p.v.}}{\pi } \int \frac {\gamma}{2}\frac {\vec n(s)}{\xi(s)-\xi(\beta)}d\beta \right)+\frac {\gamma}{2} =(I-\mathcal{K}^*)\bigl(\frac {\gamma}{2} \bigr).
\end{equation} Thus \eqref{eq-tang-velocty} follows.

For \eqref{eq-aa25}, we recall that $z(\alpha,t)=\xi(\ka(\xi,t),t)$;  hence
\begin{equation}\label{eq-a24}
\ka_t\circ \ka^{-1}e^{i\theta}=z_t\circ \ka^{-1}-\xi_t,
\end{equation} then
\begin{equation}\label{eq-a25}
\delta =(z_t\circ\ka^{-1}-\xi_t)e^{-i\theta}=\operatorname{Re} (\bar{z}_t\circ \ka^{-1}e^{i\theta})-T=(I-\mathcal{K}^*)\bigl(\frac {\gamma}{2} \bigr) -T.
\end{equation} Thus \eqref{eq-aa25} follows.

From \eqref{eq-a24}, we know that the normal velocity is invariant under the coordinate change, i.e., $U=\operatorname{Re}\left(\bar{z}_t\circ \ka^{-1} \vec{n} \right)$ with $\vec{n}=ie^{i\theta}$. Thus \eqref{eq-normal-velocity} follows.

Let us derive \eqref{eq-tangential-derivative}. We differentiate both sides of the equation $\xi_t =U \vec n+T \vec t $ to obtain
$$ \theta_t \vec n =U_s \vec n-U\theta_s \vec t+T_s \vec t+T\theta_s\vec n.$$
Then \eqref{eq-tangential-derivative} follows from multiplying it by $\vec t$.
\end{proof}

Next we will derive equations for  $\delta_t$ and $\gamma_t$.

\begin{lemma}\label{le-3}
\begin{align}
\label{eq-b14} \delta_t& =\theta_{ss}-\delta \delta_s -\sin \theta+U\theta_t-T_t.\\
\label{eq-b15} (I-K^*)(\gamma_t/2) &=[\partial_t, K^*](\gamma/2)+\theta_{ss}-\delta \delta_s -\sin \theta+U\theta_t.
\end{align}
\end{lemma}

\begin{proof}
To show \eqref{eq-b14}, we recall that, by definition, $\delta=(z_{t}\circ \ka^{-1}-\xi_t)\cdot \vec t$, i.e.,
$$ \delta = \operatorname{Re} \bigl( \bar{z}_t\circ\ka^{-1}e^{i\theta}\bigr) -T. $$
So we need to establish
\begin{equation}\label{eq-b13}
\partial_t \operatorname{Re} \bigl( \bar{z}_t\circ\ka^{-1}e^{i\theta}\bigr) =\theta_{ss}-\delta \delta_s -\sin \theta+U\theta_t.
\end{equation}
Since $\partial_t (\ka^{-1})=-\dfrac {\ka_t\circ \ka^{-1}}{\ka_\alpha\circ \ka^{-1}}=-\dfrac {\delta}{\ka_\alpha\circ \ka^{-1}}$,
\begin{equation}\label{eq-b16}
\begin{split}
 \partial_t \operatorname{Re} \bigl( \bar{z}_t\circ\ka^{-1}e^{i\theta}\bigr) & =\left(\frac {z_{t\alpha}}{|z_\alpha|}\circ \ka^{-1}\delta +z_{tt}\circ \ka^{-1}\right)\cdot \vec{t}+\operatorname{Re}(\bar{z}_t\circ \ka^{-1}\vec n)\theta_t\\
 &=\left(-\delta \partial_s(\partial_t+\delta\partial_s)\xi +z_{tt}\circ \ka^{-1}\right)\cdot \vec{t}+U\theta_t.
 \end{split}
 \end{equation}
From $\partial \xi=e^{i\theta}$, we have
\begin{equation}\label{eq-b17}
\partial_s(\partial_t+\delta\partial_s)\xi =\theta_t \vec n+ \delta_s \vec t+\delta \theta_s \vec n.
\end{equation}
This implies that $\left(-\partial_s(\partial_t+\delta\partial_s)\xi\right)\cdot \vec{t} =-\delta \delta_s.$
On the other hand, since $z_{tt}\circ \ka^{-1}=-i-\nabla P$ and $\nabla P\cdot \vec t=\partial_s P$,
\begin{equation}\label{eq-b18}
\left(z_{t}\circ \ka^{-1}\right)\cdot \vec{t}=-\sin \theta+\theta_{ss}.
\end{equation}
Thus \eqref{eq-b13} follows from \eqref{eq-b17} and \eqref{eq-b18}.

Finally \eqref{eq-b15} follows from \eqref{eq-b13} and the identity $\operatorname{Re} \bigl( \bar{z}_t\circ\ka^{-1}e^{i\theta}\bigr) =(I-K^*)\frac {\gamma}{2}$.
\end{proof}

\end{document}